\documentclass[review]{elsarticle}

\usepackage{lineno,hyperref}
\usepackage{dsfont}
\usepackage{tikz}
\usepackage{amsmath,amssymb,amsfonts,amsthm}

\usepackage{algorithm}
\usepackage{bm}
\usepackage{bookmark}
\usepackage{bm}

\newtheorem{thm}{Theorem}[section]
\newtheorem{lem}[thm]{Lemma}

\theoremstyle{definition}
\newtheorem{dfn}{Definition}[section]
\newdefinition{rmk}{Remark}[section]

\theoremstyle{remark}

\journal{Journal of \LaTeX\ Templates}

\makeatletter
\def\ps@pprintTitle{%
   \let\@oddhead\@empty
   \let\@evenhead\@empty
   \def\@oddfoot{\reset@font\hfil\thepage\hfil}
   \let\@evenfoot\@oddfoot
}
\makeatother

\begin{document}

\begin{frontmatter}

\title{On the distribution of divisors of monic polynomials in function fields}
\author[]{Yiqin He\corref{cor1}}
\ead{2014750113@smail.\;xtu.\;edu.\;cn}
\author[]{Bicheng Zhang\corref{cor1}}
\ead{zhangbicheng@xtu.\;edu.\;cn}
\address{Department of Mathematics and Computational Science, Xiangtan Univerisity, Xiangtan, Hunan, 411105, PR China}
\cortext[cor1]{Corresponding author}



%

\begin{abstract}
  This paper deals with function field analogues of famous theorems of Laudau and R. Hall.\;Laudau \cite{LA} counted the number of integers which have $t$ prime factors and R. Hall \cite{HA} researched the distribution of divisors of integers in residue classes.\;We extend the Selberg-Delange method to handle the following problems.\;(a) The number of monic polynomials with degree $n$ have $t$ irreducible factors.\;(b) The number of monic polynomials with degree $n$ in some residue classes have $t$ irreducible factors.\;(c) The residue classes distribution of divisors of monic polynomials.\;

\end{abstract}
\end{frontmatter}
\section{Introduction}
    In 1909,\;Laudau \cite{LA} counted the number of integers which have $t$ prime factors by using prime number theorem.\;We have
   \[N_k(x)=|\{n\leq x:\Omega(n)=k\}|\sim \frac{x}{\log(x)}\frac{(\log\log(x))^{k-1}}{(k-1)!},x\rightarrow +\infty \tag{1}\]
Another way to solve this problem,\;devised by Selberg (1954) by identifying $N_k(x)$ as the coefficient of $z_k$ in the expression
\[\sum_{n\leq x}z^{\Omega(n)}\]
and then applying Cauchy's integral formula.\;We use notation $a\ll_{c_1,\;c_2,\cdots} b$ or $a=O_{c_1,\;c_2,\cdots}(b)$ to mean that $|a|\leq C|b|$ for a suitable positive constant $C$ which depend upon parameters $c_1,\;c_2,\cdots$.\;

This method is nowadays known as the Selberg-Delange method,\;we refer the readers to \cite{TE} for an excellent exposition of this theory.\;Indeed,\;by this method,\;we can obtain a good estimate for the sum $\sum\limits_{n\leq x}z^{\Omega(n)}$(can see \cite{TE} $301-305$),\;that is
\[\sum_{n\leq x}z^{\Omega(n)}=x(\log(x))^{z-1}\left(\sum_{0\leq k\leq N}\frac{\nu_k(z)}{(\log(x))^k}+O_A(R_N(x))\right) \tag{2}\]
with
\[R_N(x):=e^{-c_1\sqrt{\log(x)}}+\left(\frac{c_2N+1}{\log(x)}\right)^{N+1}\]
and some coefficients $\nu_k(z)$.\;From this formula,\;we  can get an explicit asymptotic formulae for $N_k(x)$.\;Let
\[\nu(z)=\frac{1}{\Gamma(z+1)}\prod_p\left(1-\frac{z}{p}\right)^{-1}\left(1-\frac{1}{p}\right)^z,\]
then,\;we have
\[N_k(x)=\frac{x}{\log(x)}\frac{(\log\log(x))^{k-1}}{(k-1)!}\left\{\nu\left(\frac{k-1}{\log\log(x)}\right)+O\left(\frac{k}{(\log\log(x))^2}\right)\right\}\]
for $k\leq (2-\delta)\log\log(x)$,\;$0<\delta<1$,\;and
\[N_k(x)=\frac{1}{4}\prod_{p>2}\left(1+\frac{1}{p(p-2)}\right)\frac{x\log(x)}{2^k}\left(1+O\left((\log(x))^{\frac{\delta^2}{5}}\right)\right)\]
for $(2+\delta)\log\log(x)\leq k\leq A\log\log(x)$.\;

Let $\mathbb{F}_q$ be a finite field with $q$ elements,\;where
$q=p^f$,\;$p$ is the characteristic of $\mathbb{F}_q$.\;Let $\mathbf{A}=\mathbb{F}_q[T]$,\;the polynomial ring over $\mathbb{F}_q$.

In this paper, we extend the Selberg-Delange method to handle the following problems.\\
(\uppercase\expandafter{\romannumeral1})\;The number of monic polynomials of degree $n$ which have $t$ irreducible factors,\;i.e.$$N_{t}(n)=|\{f\in \mathbb{F}_q[T]:f \text{\;monic},deg(f)=n,\Omega(f)=t\}|.$$
(\uppercase\expandafter{\romannumeral2})\;The number of monic polynomials of degree $n$ which have $t$ irreducible factors and belong to some residue class,\;i.e.$$N_{t}(n;h,Q)=|\{f\in \mathbb{F}_q[T]:f \text{\;monic},deg(f)=n,\Omega(f)=t,f\equiv h(\text{mod\;}Q)\}|,$$
where  $Q\in \mathbf{A}$ is a ploynomial and $h$ is a ploynomial such that $(Q,h)=1$ and $deg(h)\leq deg(Q)$.\;

Let $\mathbf{A}_n=\{f\in \mathbf{A}|\;deg(f)=n,f \text{\;monic}\}$,\;$\mathbf{A}_n(h,Q)=\{f\in \mathbf{A}_n|\;f\equiv h(\text{mod\;}Q)\}$,\;$\mathbf{A}_n(t)=\{f\in \mathbf{A}_n|\;\Omega(f)=t\}$ and $\mathbf{A}_n(t;h,Q)=\{f\in \mathbf{A}_n|\;\Omega(f)=t,f\equiv h(\text{mod\;}Q)\}$.\;Then
we have $N_{t}(n)=|\mathbf{A}_t(n)|$ and $N_{t}(n;h,Q)=|\mathbf{A}_t(n;h,Q)|$.

To solve these problems,\;we need to define the analogue of Riemann zeta function and Drichlet $L$-function.\;Let $f\in \mathbf{A}$,\;if $f\neq0$,\;set $|f|=|q|^{deg(f)}$,\;if $f=0$,\;set $|f|=0$.\;The zeta function of $\mathbf{A}$,\;denoted by $\zeta_A(s)$,\;is defined by
\[\zeta_A(s)=\sum_{f\text{\;monic}}\frac{1}{|f|^s}=\frac{1}{1-q^{1-s}}.\]
for $s\in \mathbb{C}$ with $\mathbf{R}(s)>1$,where $\mathbf{R}(s)$ denote the real part of $s$.\;$\zeta_A(s)$ can be continued to a meromorphic function on the whole complex plane with simple poles at $s=1+\frac{2\pi in}{\log(q)}$ for $n\in \mathbb{Z}$.\;The Riemann zeta function is also a meromorphic function on the whole complex plane, which is holomorphic everywhere except for a simple pole at $s = 1$ with residue 1.\;It is difficult to solve problem (\uppercase\expandafter{\romannumeral1}) and (\uppercase\expandafter{\romannumeral2}) by using Perron fomula since the distribution of poles of Riemann zeta function and $\zeta_A(s)$ are different.

For problem (\uppercase\expandafter{\romannumeral1}),\;we use Selberg-Delange method,\;replacing Perron formula with Cauchy integral formula.\;Then,\;we obtain the following theorems.\;
\begin{thm}Let $y$ be a complex number such that $|y|\leq \rho,\rho<1$ and $n>2$ be an integer,\;then we have
\[\sum_{f\in \mathbf{A}_n}y^{\Omega(f)}=\prod_P\left(1-\frac{1}{|P|}\right)^y\left(1-\frac{y}{|P|}\right)^{-1}\frac{\kappa(y) q^{n}}{\Gamma(y)n^{1-y}}+O_{\delta,\rho}\left(\frac{q^{n}}{n^{2-\mathbf{R}(y)}}\right),\tag{3}\]
where
\[\kappa(y)=\exp\left(-(1-y)\int_{0}^{+\infty}\frac{B_1(t)dt}{(n-1+y+t)(n+t)}\right).\]
\end{thm}
\begin{thm}The number of monic polynomials of degree $n$,\;which have exactly \texorpdfstring{$t$}{Lg} irreducible factors is
\[N_{t}(n)=\sum_{f\in \mathbf{A}_n(t)}1=\frac{q^n\log^{t-1}(n)}{n}\sum_{r=1}^{t}\frac{A_r\log^{1-r}(n)}{(t-r)!}+O_{{\delta,\rho}}\left(\frac{\rho^{-n}q^n}{n^{2-\mathbf{R}(y)}}\right)\tag{4}\]
for any $\rho<1$ and $n>2$.
\end{thm}
Note that if we set $x=q^n$ in the right-hand side of equations (3) and (4),\;then they are very similar to the equations (2) and (1).\;If $n=1$,\;then equation (4) becomes
\[N_{1}(n)=\frac{q^n}{n}+O_{\delta,\rho}\left(\frac{\rho^{-n}q^{n}}{n^{2-\mathbf{R}(y)}}\right),\]
since $A_1=1$.\;This result is agree with the previous formula of $N_1(n)$,\;i.e.\;the number of irreducible polynomials of degree $n$,\;
\[N_1(n)=\frac{1}{n}\sum_{d|n}\mu{(d)}q^{\frac{n}{d}}=\frac{q^n}{n}+O\left(\frac{q^{\frac{n}{2}}}{n}\right).\]

The results associates with problem (\uppercase\expandafter{\romannumeral2}) are
\begin{thm}Let $y$ be a complex number such that $|y|\leq \rho,\rho<1$ and $n>2$ be an integer,\;then we have
\[\sum_{f\in \mathbf{A}_n(h,Q) }y^{\Omega(f)}=\frac{1}{\Phi(Q)}\prod_P\left(1-\frac{1}{|P|}\right)^y\left(1-\frac{y}{|P|}\right)^{-1}\prod_{P|Q}\frac{1}{1-\frac{1}{|P|}}\frac{\kappa(y)q^{n}}{\Gamma(y)n^{1-y}}\]
\[+O_{\delta,Q,\rho}\left(\frac{q^{n}}{n^{2-\mathbf{R}(y)}}\right).\hspace{98pt}\]
where $\kappa(y)$ is defined as above.
\end{thm}
\begin{thm}The number of monic polynomials,\;which satisfy $f\equiv h \text{\;mod}(Q)$,\;\\$deg(f)=n$ and $\Omega(f)=t$ is
\[\sum_{f\in \mathbf{A}_n(t,h,Q)}1=\left(\prod_{P|Q}\frac{1}{1-\frac{1}{|P|}}\right)\frac{q^n\log^{t-1}(n)}{\Phi(Q)n}\sum_{r=1}^{t}\frac{A_r\log^{1-r}(n)}{(t-r)!}+O_{\delta,Q,\rho}\left(\frac{\rho^{-n}q^n}{n^{2-\mathbf{R}(y)}}\right)\]
for any $\rho<1$ and $n>2$.
\end{thm}
  This theorem claims that monic polynomials in $\in A_n(t)$ seem to be equally
distributed among the $\Phi(Q)$ reduced residue classes mod $Q$.

 Another problem associates with the distribution of divisors of monic polynomial over function fields is the residue classes distribution of divisors of monic polynomial.\;R.\;HALL \cite{HA} reasearched the distribution of divisors of integers in residue classes.\;In this paper,\;we research the function fields version of it.

 Let $\tau(f;h,Q)$ be the number of  those divisors of $f$ which are prime to $Q$.\;Let $\tau(f;Q)$ be the number of those divisors of $f$ which belong to the residue classes $h\;mod(Q)$.\;The residue classes distribution of divisors of monic polynomial$f$ can be described by the variance
\[\mathbb{V}[\tau(f;\circ,Q)]:=\frac{1}{\Phi(Q)}\sum_{\substack{h,(h,Q)=1\\ deg(h)\leq deg(Q)}}\left(\tau(f;h,Q)-\frac{\tau(f;Q)}{\Phi(Q)}\right)^2.\]
In this paper,\;we evaluate the following sum,\;
\[\sum_{f\in A_n}y^{\Omega(f)}\mathbb{V}[\tau(f;\circ,Q)].\tag{5}\]
For $y=1$,\;we can obtain that
\begin{thm}For integer $n$,\;we have
\begin{equation*}
\begin{aligned}
&\sum_{\substack{f\text{\;monic}\\deg(f)=n}}\mathbb{V}[\tau(f;\circ,Q)]=\frac{1}{\Phi(Q)}\sum_{\substack{f\text{\;monic}\\deg(f)=n}}\sum_{\substack{h,deg(h)\leq deg(Q)\\(h,Q)=1}}\left(\tau(f;h,Q)-\frac{\tau(f;Q)}{\Phi(Q)}\right)^2\\
&=A_{Q}(n+1)q^{n+1}+B_{Q}q^{n+1}+O_{Q,\varepsilon}(q^{n\varepsilon}(q^{1-\varepsilon}-1)^{-2}).
\end{aligned}
\end{equation*}
Where $A_{Q}$ and $B_{Q}$ are constants associated with polynomial $Q$.
\end{thm}
This Theorem shows that the average variance $\mathbb{V}[\tau(f;\circ,Q)]$ over $A_n$ is
\[A_{Q}(n+1)q+B_{Q}q.\]
For $|y|\leq \rho$,\;$\rho<1$ and $\mathbf{R}(y)\leq\frac{1}{2}$,\;We have a good estimate for equation (5) by using Selberg-Delange method,\;i.e.
\begin{thm}\label{t-1.6}Let $y$ be a complex number such that $|y|\leq \rho$ and $\rho<1$.\;If $\mathbf{R}(y)<\frac{1}{2}$,\;then
\begin{equation*}
\begin{aligned}
&\sum_{\substack{f\text{\;monic}\\deg(f)=n}}y^{\Omega(f)}\mathbb{V}[\tau(f;\circ,Q)]
=H_1\left(\frac{1}{q},y\right)\frac{\kappa(2y)q^{n}}{\Gamma(2y)n^{1-2y}}+O_{\delta,Q,\rho}\left(\frac{q^{n}}{n^{2-2\mathbf{R}(y)}}\right)\\
\end{aligned}
\end{equation*}
for integer $n>3$,\;where
\[H_1(u,y)=\left(\prod\limits_{P|Q}\frac{1}{1+yu^{deg(P)}}\right)\frac{(1-qu^2)^{y}\mathcal{N}^2(u,y)}{\Phi(Q)^2\mathcal{N}(u^2,y)}\sum_{\chi\neq\chi_0}\mathcal{L}(u,\chi,y)\mathcal{L}(u,\overline{\chi},y).\]
 If $y=\frac{1}{2}$,\;then for integer $n$,
\begin{equation*}
\begin{aligned}
\sum_{\substack{f\text{\;monic}\\deg(f)=n}}\frac{1}{2^{\Omega(f)}}\mathbb{V}[\tau(f;\circ,Q)]
&=\frac{q^{n+1}}{\Phi(Q)^2}\prod\limits_{P|Q}\frac{1}{1+\frac{1}{2|P|}}\frac{N^2(1,\frac{1}{2})}{\zeta_A(2,\frac{1}{2})}\sum_{\chi\neq\chi_0}|L(1,\chi,\frac{1}{2})|^2\\
&+O_{\delta,Q}\left(q^{(\frac{1}{2}+\delta)n}\right).\\
\end{aligned}
\end{equation*}
\end{thm}

Theorem \ref{t-1.6} enable us to obtain the residue classes distribution of divisors of monic polynomials in set $A_n(t)$.\;
\begin{thm}For integer $n>3$,\;we have
\[\sum_{f\in A_n{(t)}}\mathbb{V}[\tau(f;\circ,Q)]=\frac{q^n(2\log(n))^{t-1}}{n}\sum_{r=1}^{t}\frac{\widehat{A}_r(2\log(n))^{1-r}}{(t-r)!}+O_{\delta,Q,\rho}\left(\frac{\rho^{-n}q^n}{n^{2-\mathbf{R}(y)}}\right)\]
for any $\rho<1$.
\end{thm}

\section{Preliminary}
Let $g\in \mathbf{A}$ be a monic polynomial and $\chi$:$(\mathbb{F}_q[T]/g\mathbb{F}_q[T])^*\rightarrow \mathbb{C}^*$ is a group homomorphism from invertible elements of $\mathbb{F}_q[T]/g\mathbb{F}_q[T]$ to the non-zero complex numbers.\;The Drichlet character of modulo $g$ is defined by
\[\chi(f)=\left\{
    \begin{array}{ll}
      \chi(f(\text{mod\;}g)), & \hbox{if $gcd(f,g)=1$,} \\
      0, & \hbox{otherwise.}
    \end{array}
  \right.
\]
The Drichlet character is multiplicative function of the polynomial ring $\mathbb{F}_q[T]$.

The Drichlet $L$-function $L(s,\chi)$ associated to Drichlet character $\chi$ is defined to be
\[L(s,\chi)=\sum_{f\text{\;monic}}\frac{\chi(f)}{|f|^s}.\]
  The function $\zeta_A(s)$ and $L(s,\chi)$ satisfy the Euler product formula
\[\zeta_A(s)=\prod_{\substack{P \text{\;irreducible}\\ P \text\;{monic}}}\left(1-\frac{1}{|P|^s}\right)^{-1},\]
and
\[L(s,\chi)=\prod_{\substack{P \text{\;irreducible}\\ P \text{\;monic}}}\left(1-\frac{\chi(P)}{|P|^s}\right)^{-1}.\]
respectively,\;provide $\mathbf{R}(s)>1$.\;The  Drichlet $L$-function $L(s,\chi)$ associated to principal character $\chi_0$ is almost the same as $\zeta_A(s)$.\;Indeed,
\[L(s,\chi_0)=\prod_{\substack{P|g \text{\;irreducible}\\ P \text{\;monic}}}\left(1-\frac{1}{|P|^s}\right)\zeta_A(s).\tag{6}\]

Let $\zeta(u)=\frac{1}{1-qu}$ and
\[\mathcal{L}(u,\chi)=\sum_{f\text{\;monic}}\chi(f)u^{deg(f)}.\]
We have $\zeta_A(s)=\zeta(q^{-s})$ and $L(s,\chi)=\mathcal{L}(q^{-s},\chi)$.\;

 For non-principal character $\chi$,\;we know that $\mathcal{L}(u,\chi)$ is a ploynomial of $u$ of degree at most $deg(g)-1$ and the Generalized Riemann Hypothesis (GRH) states that the all roots of $\mathcal{L}(u,\chi)$ have modulus $1$ ot $q^{-\frac{1}{2}}$.\;Hence,\;we have
\[\mathcal{L}(u,\chi)=\prod_{i=1}^{m(\chi)}(1-\alpha_i(\chi)u).\]
where $|\alpha_i(\chi)|=1$ or $q^{\frac{1}{2}}$ for $1\leq i\leq m(\chi)<deg(g)$.

At first,\;we estimate $\zeta(u)$ and $\mathcal{L}(u,\chi)$.\;We have $|\zeta(u)|=|\frac{1}{1-qu}|\leq \frac{1}{|1-q|u||}$ for $u\neq q^{-1}$ and for  non-principal character $\chi$ of modulo $g$,\;we have
\[|\mathcal{L}(u,\chi)|=\left|\prod_{i=1}^{m(\chi)}(1-\alpha_i(\chi)u)\right|\leq (1+\sqrt{q}|u|)^{m}\tag{7}\]
where $m=deg(g)-1$.\;For any $|y|\leq \rho$,\;$\rho<1$,\;let \[\zeta_A(s,y)=\prod_P\left(1-\frac{y}{|P|^s}\right)^{-1}=\sum_{f\text{\;monic}}\frac{y^{\Omega(f)}}{|f|^s},\]
\[L(s,\chi,y)=\prod_P\left(1-\frac{y\chi(P)}{|P|^s}\right)^{-1}=\sum_{f\text{\;monic}}\frac{\chi(f)y^{\Omega(f)}}{|f|^s},\]
where $\Omega(f)$ denotes the number of prime divisors of $f$(Including multiplicities).
 Let
 \[N(s,y)=\zeta_A(s)^{-y}\prod_P\left(1-\frac{y}{|P|^s}\right)^{-1}=\prod_P\frac{\left(1-\frac{1}{|P|^s}\right)^y}{1-\frac{y}{|P|^s}}\]
 and
  \[M(s,\chi,y)=L(s,\chi)^{-y}\prod_P\left(1-\frac{y\chi(P)}{|P|^s}\right)^{-1}=\prod_P\frac{\left(1-\frac{\chi(P)}{|P|^s}\right)^y}{1-\frac{y\chi(P)}{|P|^s}}.\]
Then,\;$N(s,y)$ and $M(s,\chi,y)$ are convergent and bounded in the half plane $\mathbf{R}(s)\geq \frac{1}{2}+\delta $ for any $\delta>0$.\;Indeed,\;Let
\[\frac{(1-z)^y}{1-yz}=1+\sum_{v\geq 1}a_vz^v,\]
then $a_1=0$ and
\[|a_v|=\left|\frac{1}{2\pi i}\int_{C_r}\frac{(1-z)^y}{1-yz}z^{-v-1}dz\right|\leq \frac{M(\rho)}{r^v}\]
for $v\geq 2$ and $r=\frac{1}{2\rho}$.\;Where
\[M(\rho)=\max_{|z|=r,|y|\leq\rho}\left|\frac{(1-z)^y}{1-yz}\right|\leq2e^{\rho(\pi+1+r)}=2e^{\rho(\pi+1)+\frac{1}{2}}.\]

Let $z=\frac{1}{|P|^s}$,\;then
\begin{equation*}
\begin{aligned}
& N(s,y)=\prod_P\frac{\left(1-\frac{1}{|P|^s}\right)^y}{1-\frac{y}{|P|^s}}=\prod_P\left(1+\sum_{v\geq 2}\frac{a_v}{|P|^{vs}}\right)\\
& \ll_{\delta,\rho} \prod_P\left(1+\sum_{v\geq 2}\frac{M(\delta)}{r^v|P|^{v(\frac{1}{2}+\delta)}}\right)\ll_{\delta,\rho} \prod_P\exp\left(\sum_{v\geq 2}\frac{1}{r^v|P|^{v(\frac{1}{2}+\delta)}}\right)\\
& =\exp\left(\sum_P\sum_{v\geq 2}\frac{1}{(r|P|^{(\frac{1}{2}+\delta)})^v}\right)=\exp\left(\sum_P\sum_{v\geq 2}\frac{1}{(r|P|^{(\frac{1}{2}+\delta)})^2-r|P|^{(\frac{1}{2}+\delta)}}\right)\\
& \leq\exp\left(\sum_{n\geq 1}\frac{q^n/n}{r^2q^{n(1+2\delta)}-rq^{n(\frac{1}{2}+\delta)}}\right)\ll_{\delta,\rho} 1\\
\end{aligned}\eqno{(8)}
\end{equation*}
 We can also get $M(s,\chi,y)\ll_{\delta,\rho} 1$.

Putting variable substitution $u=q^{-s}$ in $\zeta_A(s,y)$,\;$L(s,\chi,y)$,\;$N(s,y)$ and $M(s,\chi,y)$,\;we denote these new functions with
$\zeta(u,y)$,\;$\mathcal{L}(u,\chi,y)$,\;$\mathcal{N}(u,y)$ and $\mathcal{M}(u,\chi,y)$ respectively.\;Indeed,\;we have
\[\zeta(u,y)=\prod_P\frac{1}{1-yu^{deg(P)}}=\sum_{f\text{\;monic}}y^{\Omega(f)}u^{deg(f)},\]

\[\mathcal{L}(u,\chi,y)=\prod_P\frac{1}{1-y\chi{(P)}u^{deg(P)}}=\sum_{f\text{\;monic}}y^{\Omega(f)}\chi(f)u^{deg(f)},\]
and
\[\mathcal{N}(u,y)=\prod_P\frac{\left(1-u^{deg(P)}\right)^y}{1-yu^{deg(P)}},\mathcal{M}(u,\chi,y)=\prod_P\frac{\left(1-\chi{(P)}u^{deg(P)}\right)^y}{1-y\chi{(P)}u^{deg(P)}}\]

Then $\mathcal{N}(u,y),\mathcal{M}(u,\chi,y)$ are convergent and bounded if $|u|\leq \frac{1}{q^{\frac{1}{2}+\delta}}$.\;Note that
$\mathcal{L}(u,\chi)$ is a holomorphic function and has zeros only on the circle $|u|=1$ or $q^{-\frac{1}{2}}$(by $GRH$),\;so
$\mathcal{L}(u,\chi,y)=\mathcal{L}(u,\chi)^y\mathcal{M}(u,\chi,y)$ is holomorphic function on the disc $\{u:|u|\leq \frac{1}{q^{\frac{1}{2}+\delta}}\}$.
$\zeta(u)$ is a meromorphic function on the whole complex plane,\;which
is holomorphic everywhere except for a simple pole at $u=\frac{1}{q}$,\;$\zeta(u)$ has no zero.\;Thus $u=\frac{1}{q}$ may be the pole of function $\zeta(u,y)=\zeta(u)^y\mathcal{N}(s,y)$.

we now give the bound of $\zeta(u,y)$ and $\mathcal{L}(u,\chi,y)$ on the disc $\{u:|u|\leq \frac{1}{q^{\frac{1}{2}+\delta}}\}$.\;Indeed,\;we have
\begin{equation*}
\begin{aligned}
\zeta(u,y)&=e^{-y(\log|1-qu|+i\arg(1-qu))}\mathcal{N}(u,y)\ll_\delta e^{-\mathbf{R}(y)\log|1-qu|+\mathbf{I}(y)\arg(1-qu)}\\
&\ll_{\delta,\rho}|1-qu|^{-\mathbf{R}(y)},~~~~~~~~~~~~~~~~~~~~~~~~~~~~~~~~~~~~~~~~~~~~~~~~~~~~~~~~~~~~~(9)\\
\mathcal{L}(u,\chi,y)&=\prod_{i=1}^m(1-\alpha_i(\chi)z)^y\mathcal{M}(u,\chi,y)\\
& \ll_{\delta,\rho} \prod_{i=1}^me^{\mathbf{R}(y)\log|1-\alpha_i(\chi)u|-\mathbf{I}(y)\arg(1-\alpha_i(\chi)u)}\\
& \ll_{\delta,\rho} (1+\sqrt{q}|u|)^{\mathbf{R}(y)m}.~~~~~~~~~~~~~~~~~~~~~~~~~~~~~~~~~~~~~~~~~~~~~~~~~~~~~~~(10)\\
\end{aligned}
\end{equation*}

\section{The number of monic polynomials of degree \texorpdfstring{$n$ which have $t$}{Lg} irreducible factors}
 Let $B_n(t)$ denote $n$-th Bernoulli function as the 1-periodic
function.\;We begin this section with a estimate of $\log\Gamma(z)$.\;
\begin{lem}(Complex Stirling formula)For $s\in \mathbb{C}-\mathbb{R}^-$,\;we have
\[\log \Gamma(s)=\left(s-\frac{1}{2}\right)\log(s)-s+\frac{1}{2}\log (2\pi)-\int_{0}^{+\infty}\frac{B_1(t)dt}{s+t},\tag{11}\]
where the complex logarithm is understood as its
principal branch,\;i.e.$\log(s)=\exp(log|s|+i\arg(s))$ and $-\pi<\arg(s)\leq\pi$.
\end{lem}
\begin{dfn}(Beta function) For complex numbers $x$,\;$y$ satisfy $\mathbf{R}(x),\mathbf{R}(y)>0$,\;we define
\[B(x,y)=\int_{0}^{\infty}t^{x-1}(1-t)^{y-1}dt.\]
The relationship between gamma function and beta function is
\[B(x,y)=\frac{\Gamma(x)\Gamma(y)}{\Gamma(x+y)}.\tag{12}\]
\end{dfn}
\begin{lem}Let $y$ be a complex number such that $|y|\leq \rho,\rho<1$ and $n>1$ be an integer.\;We have
\[\frac{\Gamma(n-1+y)}{\Gamma(n)}=\frac{\kappa(y)}{n^{1-y}}+O\left(\frac{1}{n^{2-\mathbf{R}(y)}}\right),\]
where \[\kappa(y)=\exp\left(-(1-y)\int_{0}^{+\infty}\frac{B_1(t)dt}{(n-1+y+t)(n+t)}\right).\]
\end{lem}
\begin{proof}From equation (11),\;we have
\begin{equation*}
\begin{aligned}
\log\frac{\Gamma(n-1+y)}{\Gamma(n)}&=\left(n-\alpha-\frac{1}{2}\right)\log(n-\alpha)-(n-\alpha)+\frac{1}{2}\log (2\pi)-\int_{0}^{+\infty}\frac{B_1(t)dt}{n-\alpha+t}\\
&-\left(n-\frac{1}{2}\right)\log(n)+n-\frac{1}{2}\log (2\pi)+\int_{0}^{+\infty}\frac{B_1(t)dt}{n+t}\\
&=\left(n-\alpha-\frac{1}{2}\right)\log(n-\alpha)-\left(n-\frac{1}{2}\right)\log(n)+C_y\\
\end{aligned}
\end{equation*}
where $\alpha=1-y$ and $C_y=(1-y)\left(1-\int_{0}^{+\infty}\frac{B_1(t)dt}{(n-1+y+t)(n+t)}\right)$.
 Thus
\[\frac{\Gamma(n-1+y)}{\Gamma(n)}=e^{C_y}\frac{(n-\alpha)^{n-\alpha-\frac{1}{2}}}{n^{n-\frac{1}{2}}}=e^{C_y}n^{-\alpha}\left(1-\frac{\alpha}{n}\right)^{n-\alpha-\frac{1}{2}}\]
\[=\frac{\kappa(y)}{n^{1-y}}+O_{\rho}\left(\frac{1}{n^{2-\mathbf{R}(y)}}\right),\hspace{25pt}\]
where $\kappa(y)=\exp\left(-(1-y)\int_{0}^{+\infty}\frac{B_1(t)dt}{(n-1+y+t)(n+t)}\right)$.
\end{proof}
\begin{thm}\label{t-3.3}Let $y$ be a complex number such that $|y|\leq \rho,\rho<1$ and $n>2$ be an integer.\;Then we have
\[\sum_{\substack{f\text{\;monic}\\deg(f)=n}}y^{\Omega(f)}=\prod_P\left(1-\frac{1}{|P|}\right)^y\left(1-\frac{y}{|P|}\right)^{-1}\frac{\kappa(y) q^{n}}{\Gamma(y)n^{1-y}}+O_{\delta,\rho}\left(\frac{q^{n}}{n^{2-\mathbf{R}(y)}}\right)\]
where $\kappa(y)$ is defined as above.
\end{thm}

\begin{proof}
Let $|y|\leq \rho,\rho<1$.\;Note that
\[\zeta_A(s,y)=\sum_{f\text{\;monic}}\frac{y^{\Omega(f)}}{|f|^s}.\]
Thus,\;let $u=q^{-s}$,\;for $s$ with $\mathbf{R}(s)>1$,\;we have
\[\sum_{n\geq 0}\left(\sum_{\substack{f\text{\;monic}\\deg(f)=n}}y^{\Omega(f)}\right)u^n=\zeta(u,y)=\zeta(u)^y\mathcal{N}(u,y)=\frac{\mathcal{N}(u,y)}{(1-qu)^y}.\]
provide $|u|<\frac{1}{q}$.\;By the Cauchy integral formula,\;we have
\[\sum_{\substack{f\text{\;monic}\\deg(f)=n}}y^{\Omega(f)}=\frac{1}{2\pi i}\int_{C_R}\frac{\mathcal{N}(z,y)}{z^{n+1}(1-qz)^y}dz,\]
where $C_R$ is a circle centered at $0$ with radius $R$.\;As shown in Figure $1$,\;we consider the region $\Omega_1$ of integration which surrounded by contour $\gamma_1=\widetilde{C}_{q^{-\frac{1}{2}-\delta}}\cup \widetilde{C}_{\epsilon}\cup \Gamma_1 \cup \Gamma_2 \cup C_R $.\; $\gamma_1$ is defined by the following curves,\;
\begin{equation*}
\begin{aligned}
\widetilde{C}_{q^{-\frac{1}{2}-\delta}}:&\text{\;semi circle centered at\;} 0 \text{\;with radius\;} q^{-\frac{1}{2}-\delta} \\
&\text{\;from\;} q^{-\frac{1}{2}-\delta}e^{i\beta} \text{\;to\;} q^{-\frac{1}{2}-\delta}e^{i(2\pi-\beta)}.\\
\widetilde{C}_{\epsilon}:&\text{\;semi circle centered at\;} z_0=\frac{1}{q} \text{\;with radius\;} \epsilon \\
&\text{\;from\;} z_0+\epsilon e^{i(2\pi-\alpha)} \text{\;to\;}z_0+\epsilon e^{i\alpha}.\\
\Gamma_1:& \text{\;the straight line from\;} z_0+\epsilon e^{i\alpha}  \text{\;to\;}  q^{-\frac{1}{2}-\delta}e^{i\beta}. \\
\Gamma_2:& \text{\;the straight line from\;}  q^{-\frac{1}{2}-\delta}e^{i(2\pi-\beta)} \text{\;to\;} z_0+\epsilon e^{i(2\pi-\alpha)}.\\
C_R:&\text{\; circle centered at\;} 0 \text{\;with radius\;} R.\\
\end{aligned}
\end{equation*}
\begin{figure}
 \centering
\begin{tikzpicture}[domain=-4:4]
\draw[->] (-3.7,0) -- (3.7,0) node[below] {$t$};
\draw[->] (0,-3.7) -- (0,3.7) node[above] {$\sigma$};
\draw[thick] (0,0)circle (0.8);
\draw[thick] (2.027,0.259) arc (32: 230: 0.5);
\draw[thick] (1.173,-0.259) arc (212: 328: 0.5);
\draw[thick] (2.989,0.259) arc (5: 355: 3.0);

%
\draw[-][thick] (2.027,0.259) -- (3.0,0.259);
\draw[-][thick] (2.027,-0.259) -- (3.0,-0.259);

\fill[thick] (1.6,0)circle (1pt);
\node[above] at (2.5,0.4){$\Gamma_1$};
\node[above] at (2.5,-0.8){$\Gamma_2$};
\node[above] at (2.3,2.3){$C_{\frac{1}{q^{\frac{1}{2}+\delta}}}$};
\node[above] at (-0.6,0.6){$C_R$};
\node[above] at (1.5,0.5){$C_\epsilon$};
\node[above] at (1.6,0){$z_0$};
\draw[->][thick] (2.7,0.259) -- (2.8,0.259);
\draw[<-][thick] (2.7,-0.259) -- (2.8,-0.259);
\node[above] at (-1.7,-1.7){$\Omega_1$};
\end{tikzpicture}
\caption{Region $\Omega_1$ and contour $\gamma_1$}
\end{figure}
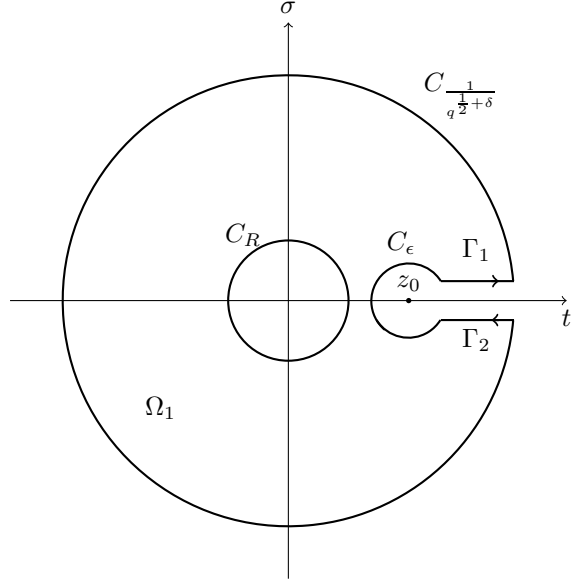

We shall see that the main contribution arises from the integral over the straight lines $\Gamma_1$ and $\Gamma_2$.\;From (8),\;the integral over $\widetilde{C}_{q^{-\frac{1}{2}-\delta}}\triangleq \widetilde{C}$ does not exceed
\begin{equation*}
\begin{aligned}
  \ll_{\delta,\rho}\int_{\widetilde{C}}\frac{1}{|z|^{n+1}(q^{\frac{1}{2}-\delta}-1)^{\mathbf{R}(y)}}dz=\frac{ q^{n(\frac{1}{2}+\delta)}}{(q^{\frac{1}{2}-\delta}-1)^{\mathbf{R}(y)}}.\\
  \end{aligned}
\end{equation*}
The integral over $\widetilde{C}_{\epsilon}$ does not exceed
\[\frac{1}{2\pi i}\int_{\widetilde{C}_\epsilon^-}\frac{\mathcal{N}(z,y)}{z^{n+1}(1-qz)^y}dz=\frac{\epsilon^{1-y}}{2\pi q^{y}}\int_{\alpha}^{2\pi-\alpha}\frac{e^{(1-y)i\theta-\pi yi}\mathcal{N}(\frac{1}{q}+\epsilon e^{i\theta},y)}{(\frac{1}{q}+\epsilon e^{i\theta})^{n+1}}d\theta\]
\[\ll_{\delta,n,\rho}\epsilon^{1-\mathbf{R}(y)}.\]
Then this integral tends to zero with $\epsilon\rightarrow0$ since $\mathbf{R}(y)<1$.\;Letting the straight lines $\Gamma_1$,\;$\Gamma_2$ onto the real line and $\epsilon\rightarrow0$ together,\;we have
\begin{equation*}
\begin{aligned}
\frac{1}{2\pi i}\int_{C_R}\frac{\mathcal{N}(z,y)}{z^{n+1}(1-qz)^y}dz
&=\frac{q^{-y}}{2\pi i}\int_{\frac{1}{q}}^{\frac{1}{q^{\frac{1}{2}+\delta}}}\frac{\mathcal{N}(u,y)}{u^{n+1}}e^{-y(\log(u-\frac{1}{q})-i\pi)}du\\
&-\frac{q^{-y}}{2\pi i}\int_{\frac{1}{q}}^{\frac{1}{q^{\frac{1}{2}+\delta}}}\frac{\mathcal{N}(u,y)}{u^{n+1}}e^{-y(\log(u-\frac{1}{q})+i\pi)}du\\
&+O_{\delta,\rho}\left(\frac{ q^{n(\frac{1}{2}+\delta)}}{(q^{\frac{1}{2}-\delta}-1)^{\mathbf{R}(y)}}\right),\\
\end{aligned}
\end{equation*}
and therefore,\;we have
\begin{equation*}
\begin{aligned}
\sum_{\substack{f\text{\;monic}\\deg(f)=n}}y^{\Omega(f)}&=\frac{\sin(y\pi)}{q^{y}\pi}\int_{\frac{1}{q}}^{\frac{1}{q^{\frac{1}{2}+\delta}}}\frac{\mathcal{N}(u,y)}{u^{n+1}}(u-\frac{1}{q})^{-y}du+O_{\delta,\rho}\left(\frac{ q^{n(\frac{1}{2}+\delta)}}{(q^{\frac{1}{2}-\delta}-1)^{\mathbf{R}(y)}}\right)\\
&=\frac{\sin(y\pi)}{q^{y}\pi}\int_{0}^{\omega_0}\frac{\mathcal{N}(\frac{1}{q}+\omega,y)}{(\frac{1}{q}+\omega)^{n+1}}\omega^{-y}d\omega+O_{\delta,\rho}\left(\frac{ q^{n(\frac{1}{2}+\delta)}}{(q^{\frac{1}{2}-\delta}-1)^{\mathbf{R}(y)}}\right).\\
\end{aligned}
\end{equation*}
where $\omega_0=\frac{1}{q^{\frac{1}{2}+\delta}}-\frac{1}{q}$.\;By Cauchy's theorem.\;We have
\[\frac{\mathcal{N}(\frac{1}{q}+\omega,y)}{\frac{1}{q}+\omega}-q\mathcal{N}\left(\frac{1}{q},y\right)=\frac{\omega}{2\pi i}\int_{D}\frac{\mathcal{N}(z,y)}{z(z-\frac{1}{q}-\omega)(z-\frac{1}{q})}dz=\omega k(\omega),\]
where $D$ is circle,\;centre $\frac{1}{q}+\frac{\omega_0}{2}$ and radius $\omega_0$.\;Thus
\begin{equation*}
\begin{aligned}
&\frac{\sin(y\pi)}{q^{y}\pi}\int_{0}^{\omega_0}\frac{\mathcal{N}(\frac{1}{q}+\omega,y)}{(\frac{1}{q}+\omega)^{n}}\omega^{-y}d\omega\\
&=q\mathcal{N}(\frac{1}{q},y)\frac{\sin(y\pi)}{q^{y}\pi}\int_{0}^{\omega_0}\frac{\omega^{-y}}{(\frac{1}{q}+\omega)^{n}}d\omega+\frac{\sin(y\pi)}{q^{y}\pi}\int_{0}^{\omega_0}\frac{\omega^{1-y}k(\omega)}{(\frac{1}{q}+\omega)^{n+1}}d\omega\\
&=q\mathcal{N}(\frac{1}{q},y)\frac{\sin(y\pi)}{q^{y}\pi}\int_{0}^{\omega_0}\frac{\omega^{-y}}{(\frac{1}{q}+\omega)^{n}}d\omega+O_{\delta,\rho}\left(\int_{0}^{\omega_0}\frac{\omega^{1-y}}{(\frac{1}{q}+\omega)^{n}}d\omega\right).\\
\end{aligned}\eqno{(13)}
\end{equation*}

  Note that
\begin{equation*}
\begin{aligned}
  &\int_{0}^{\omega_0}\frac{\omega^{-y}}{(\frac{1}{q}+\omega)^{n}}d\omega=q^{n-1+y}\int_{0}^{q\omega_0}\frac{x^{-y}dx}{(1+x)^n}\\
  &=q^{n-1+y}\int_{0}^{+\infty}\frac{x^{-y}dx}{(1+x)^n}+O\left(\frac{(q\omega_0)^{-y}(1+q\omega_0)^{1-n}}{n-1}\right)\\
 &=q^{n-1+y}\int_{0}^{1}(1-x)^{-y}x^{n-2+y}dx+O\left(\frac{(q\omega_0)^{-y}(1+q\omega_0)^{1-n}}{n-1}\right)\\
&=q^{n-1+y}\frac{\Gamma(1-y)\Gamma(n-1+y)}{\Gamma(n)}+O\left(\frac{q^{(n-1-\mathbf{R}(y))(\frac{1}{2}-\delta)}}{n-1}\right)\\
&=q^{n-1+y}\Gamma(1-y)\frac{\kappa(y)}{n^{1-y}}+O_{\rho}\left(\frac{q^{n-1+\mathbf{R}(y)}}{n^{2-\mathbf{R}(y)}}\right).\\
 \end{aligned}
\end{equation*}
The error term of (13) does not exceed
\begin{equation*}
\begin{aligned}
O_{\delta,Q}\left(\int_{0}^{\omega_0}\frac{\omega^{1-y}}{(\frac{1}{q}+\omega)^{n}}d\omega\right)&=O_{\delta,\rho}(q^{n-2+\mathbf{R}(y)}|B(2-y,n-2+y)|)\\
&=O_{\delta,\rho}\left(\frac{q^{n}}{n^{2-\mathbf{R}(y)}}\right).
 \end{aligned}
\end{equation*}
Thus we can obtain that
\begin{equation*}
\begin{aligned}
\sum_{\substack{f\text{\;monic}\\deg(f)=n}}y^{\Omega(f)}&=\mathcal{N}\left(\frac{1}{q},y\right)\frac{\sin(y\pi)}{\pi}q^{n}\Gamma(1-y)\frac{\kappa(y)}{n^{1-y}}+O_{\delta,\rho}\left(\frac{q^{n}}{n^{2-\mathbf{R}(y)}}\right)\\
&=\prod_P\left(1-\frac{1}{|P|}\right)^y\left(1-\frac{y}{|P|}\right)^{-1}\frac{\kappa(y) q^{n}}{\Gamma(y)n^{1-y}}+O_{\delta,\rho}\left(\frac{q^{n}}{n^{2-\mathbf{R}(y)}}\right)
\end{aligned}
\end{equation*}
by the Reflection formula.
\end{proof}
Note that $\prod\limits_P\left(1-\frac{1}{|P|}\right)^y\left(1-\frac{y}{|P|}\right)^{-1}$ is convergent for $|y|\leq \rho$,\;let
\[\frac{\kappa(y)}{\Gamma(y)}\prod_P\left(1-\frac{1}{|P|}\right)^y\left(1-\frac{y}{|P|}\right)^{-1}=\sum_{r=1}^{+\infty}A_ry^r,\]
then we obtain the following theorem of this section.

\begin{thm}The number of monic polynomials of degree \texorpdfstring{$n>2$}{Lg} which have exactly \texorpdfstring{$t$}{Lg} irreducible factors is
\[\sum_{\substack{f\text{\;monic}\\deg(f)=n,\Omega(f)=t}}1=\frac{q^n\log^{t-1}(n)}{n}\sum_{r=1}^{t}\frac{A_r\log^{1-r}(n)}{(t-r)!}+O_{\delta,\rho}\left(\frac{\rho^{-n}q^n}{n^{2-\mathbf{R}(y)}}\right)\]
for any $\rho<1$.
\end{thm}
\begin{proof}

By Cauchy's coefficient formula,
\[\sum_{\substack{f\text{\;monic}\\deg(f)=n,\Omega(f)=t}}1=\frac{q^n}{n}\sum_{r=1}^{t}\frac{A_r\log(n)^{t-r}}{(t-r)!}+O_{\delta,\rho}\left(\frac{\rho^{-n}q^n}{n^{2-\mathbf{R}(y)}}\right)\]
for any $\rho<1$.\;This completes the proof.
\end{proof}

\section{The number of monic polynomials of degree \texorpdfstring{$n$ which have $t$}{Lg} irreducible factors and belong to some residue classes}
\begin{thm}\label{t-4.1}Let $y$ be a complex number such that $|y|\leq \rho,\rho<1$ and $n>2$ be an integer,\;then we have
\[\sum_{\substack{f,\text{\;monic}\\f\equiv h \text{\;mod}(Q)}}y^{\Omega(f)}=\frac{1}{\Phi(Q)}\prod_P\left(1-\frac{1}{|P|}\right)^y\left(1-\frac{y}{|P|}\right)^{-1}\prod_{P|Q}\frac{1}{1-\frac{1}{|P|}}\frac{\kappa(y)q^{n}}{\Gamma(y)n^{1-y}}\]
\[+O_{\delta,Q,\rho}\left(\frac{q^n}{n^{2-\mathbf{R}(y)}}\right).\hspace{94pt}\]
where $\kappa(y)$ is defined as above.
\end{thm}
\begin{proof}
Let $\chi$ be a Dirichlet character  of modulo $Q$ and $h$ be a ploynomial such that $(Q,h)=1$.\;From the Orthogonal Relation of Drichlet character and equation (6),\;we have
\begin{equation*}
\begin{aligned}
\sum_{\substack{f,\text{\;monic}\\f\equiv h \text{\;mod}(Q)}}\frac{y^{\Omega(f)}}{|f|^s}
&=\sum_{f\text{\;monic}}\frac{y^{\Omega(f)}}{|f|^s}\left(\frac{1}{\Phi(Q)}\sum_{\chi}\chi(f)\overline{\chi}(h)\right)\\
&=\frac{1}{\Phi(Q)}\sum_{\chi}\overline{\chi}(h)L(s,\chi,y)\\
&=\frac{1}{\Phi(Q)}L(s,\chi_0,y)+\sum_{\chi\neq\chi_0}\overline{\chi}(h)L(s,\chi,y)\\
&=\left(\prod_{P|Q}\frac{1}{1-\frac{1}{|P|^s}}\right)\frac{\zeta_A(s,y)}{\Phi(Q)}+\frac{1}{\Phi(Q)}\sum_{\chi\neq\chi_0}\overline{\chi}(h)L(s,\chi,y).\\
\end{aligned}
\end{equation*}
Thus,\;let $u=q^{-s}$,for $s$ and $\mathbf{R}(s)>1$,\;we have
\[\sum_{n\geq 0}\left(\sum_{\substack{f,\text{\;monic}\\f\equiv h \text{\;mod}(Q)}}y^{\Omega(f)}\right)u^n=\frac{1}{\Phi(Q)}\sum_{\chi}\overline{\chi}(h)\mathcal{L}(u,\chi,y)\]
provide $|u|<\frac{1}{q}$.\;By the Cauchy integral formula,\;we have
\[\sum_{\substack{f\text{\;monic}\\deg(f)=n}}y^{\Omega(f)}=\frac{1}{2\pi i\Phi(Q)}\int_{C_R}\frac{g_Q(z)\zeta(z,y)}{z^{n+1}}dz+\frac{1}{2\pi i\Phi(Q)}\int_{C_R}\sum_{\chi\neq\chi_0}\frac{\overline{\chi}(h)\mathcal{L}(z,\chi,y)}{z^{n+1}}dz\]
\[\triangleq I_1+I_2,\hspace{178pt}\tag{14}\]
where $C_R$ is a circle centered at $0$ with radius $R$.\;We consider the same region $\Omega_1$ and contour $\gamma_1$ of the Figure 1.\;The integral $I_1$ over $\widetilde{C}_{q^{-\frac{1}{2}-\delta}}\triangleq \widetilde{C}$ does not exceed
\[\ll_{\delta,Q}q^{n(\frac{1}{2}+\delta)}\left(|q^{\frac{1}{2}-\delta}-1|^{-\mathbf{R}(y)}\right).\]
since $g_Q(z)$ is holomorphic function in the region $\Omega_1$ and we have $g_Q(z)\ll_{\delta,Q}1$.\;The integral $I_1$ over $\widetilde{C}_\epsilon$ does not exceed
\begin{equation*}
\begin{aligned}
&\frac{1}{2\pi i\Phi(Q)}\int_{\widetilde{C}_\epsilon^-}\frac{g_Q(z)\mathcal{N}(z,y)}{z^{n+1}(1-qz)^y}dz\\
&=\frac{\epsilon^{1-y}}{2\pi i \Phi(Q)q^{y}}\int_{\alpha}^{2\pi-\alpha}\frac{e^{(1-y)i\theta-\pi yi}g_Q(\frac{1}{q}+\epsilon e^{i\theta})\mathcal{N}(\frac{1}{q}+\epsilon e^{i\theta},y)}{(\frac{1}{q}+\epsilon e^{i\theta})^{n+1}}d\theta\\
&\ll_{\delta,Q,\rho,n}\epsilon^{1-\mathbf{R}(y)}.
\end{aligned}
\end{equation*}
Then this integral tends to zero as $\epsilon\rightarrow0$ since $\mathbf{R}(y)<1$.\;Letting the $\Gamma_1$,\;$\Gamma_2$ onto the real line and $\epsilon\rightarrow0$ together,\;we have
\begin{equation*}
\begin{aligned}
&I_1=\frac{\sin(y\pi)}{q^{y}\pi\Phi(Q)}\int_{\frac{1}{q}}^{\frac{1}{q^{\frac{1}{2}+\delta}}}\frac{g_Q(u)\mathcal{N}(s,y)}{u^{n+1}}(u-\frac{1}{q})^{-y}du+O_{\delta,Q}\left(q^{n(\frac{1}{2}+\delta)}\right)\\
&=\frac{1}{\Phi(Q)}\prod_P\left(1-\frac{1}{|P|}\right)^y\left(1-\frac{y}{|P|}\right)^{-1}\prod_{P|Q}\frac{1}{1-\frac{1}{|P|}}\frac{\kappa(y)q^{n}}{\Gamma(y)n^{1-y}}+O_{\delta,Q,\rho}\left(\frac{q^{n}}{n^{2-\mathbf{R}(y)}}\right).
\end{aligned}\eqno{(15)}
\end{equation*}
 For integral $I_2$,\;we consider another region $\Omega_2$ surrounded by contour $\gamma_2=C_{q^{-\frac{1}{2}-\delta}}\cup C_R$ of integration as shown in the Figure 2,\\
\begin{figure}[H]
 \centering
\begin{tikzpicture}[domain=-3:3]
\draw[->] (-3.2,0) -- (3.2,0) node[below] {$t$};
\draw[->] (0,-3.2) -- (0,3.2) node[above] {$\sigma$};
\draw[thick] (0,0) circle (1.5);
\draw[<-,thick] (1.5,0)--(1.5,0.1);
\draw[->,thick] (2.8,-0.1)--(2.8,0);
\draw[->,thick] (-1.5,-0.1)--(-1.5,0);
\draw[->,thick] (-2.8,0.1)--(-2.8,0);
\draw[thick] (0,0) circle (2.8);
\fill[thick] (2.1,0)circle (1pt);
\node[above] at (2.1,2.1){$C_{q^{-\frac{1}{2}-\delta}}$};
\node[above] at (0.6,0.6){$C_R$};
\node[above] at (2.1,0){$z=\frac{1}{q}$};
\node[above] at (-1.7,-1.7){$\Omega_2$};
\end{tikzpicture}
\caption{Region $\Omega_2$ and contour $\gamma_2$}
\end{figure}
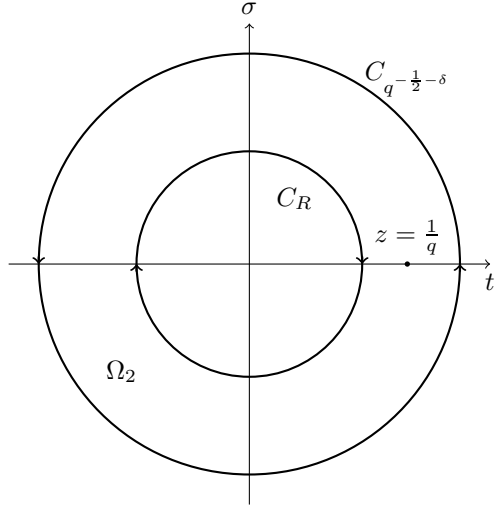
 It follows from (10) that the integral over $C_{q^{-\frac{1}{2}-\delta}}$ satisfies
 \[I_2=\frac{1}{2\pi i\Phi(Q)}\int_{C_{q^{-\frac{1}{2}-\delta}}}\sum_{\chi\neq\chi_0}\frac{\overline{\chi}(h)\mathcal{L}(z,\chi,y)}{z^{n+1}}dz\ll_{\delta,Q}q^{n(\frac{1}{2}+\delta)}(1+q^{\frac{1}{2}-\delta})^{\mathbf{R}(y)m},\tag{16}\]
since $\sum\limits_{\chi\neq\chi_0}\frac{\overline{\chi}(h)\mathcal{L}(z,\chi,y)}{z^{n+1}}$ is a holomorphic function over region $\Omega_2$.\;It follows from equations (14),\;(15) and (16) that
\[\sum_{\substack{f,\text{\;monic}\\f\equiv h \text{\;mod}(Q)}}y^{\Omega(f)}=\frac{1}{\Phi(Q)}\prod_P\left(1-\frac{1}{|P|}\right)^y\left(1-\frac{y}{|P|}\right)^{-1}\prod_{P|Q}\frac{1}{1-\frac{1}{|P|}}\frac{\kappa(y)q^{n}}{\Gamma(y)n^{1-y}}\]
\[+O_{\delta,Q,\rho}\left(\frac{q^{n}}{n^{2-\mathbf{R}(y)}}\right).\hspace{92pt}\]
\end{proof}

Applying Cauchy's integral formula to Theorem \ref{t-4.1},\;we have
\begin{thm}The number of monic polynomials of degree $n$ which satisfy $f\equiv h \text{\;mod\;}(Q)$ and $\Omega(f)=t$ is
\[\sum_{f\in A_n(t;h,Q)}1=\left(\prod_{P|Q}\frac{1}{1-\frac{1}{|P|}}\right)\frac{q^n\log^{t-1}(n)}{\Phi(Q)n}\sum_{r=1}^{t}\frac{A_r\log^{1-r}(n)}{(t-r)!}+O_{\delta,Q,\rho}\left(\frac{\rho^{-n}q^{n}}{n^{2-\mathbf{R}(y)}}\right)\]
for any $\rho<1$.
\end{thm}

\section{On the residue classes distribution of divisors of monic polynomial}

In this section,\;let $Q\in \mathbf{A}$ be a ploynomial.\;Let $h$ be a ploynomial such that $(Q,h)=1$ and $deg(h)\leq deg(Q)$,\;we reasearch the divisors of monic polynomial $f$ in residue class $h\;\text{mod\;}(Q)$.\;

For a ploynomial $f\in \mathbf{A}$ and Dirichlet character $\chi$ of modulo $Q$,\;we define
\[\sigma_a(f,\chi)=\sum_{\substack{d|f\\d \text{\;monic}}}\chi(d)|d|^a.\]
Let $\sigma_a(f;Q)$ be the sum of the $a$-th ``powers" of those divisors of $f$ which are prime to $Q$,\;i.e.,
\[\sigma_a(f;Q)=\sum_{\substack{d|f,d\text{\;monic}\\(d,Q)=1}}|d|^a.\]
Let $\sigma_a(f;Q)$ be the sum of the $a$-th ``powers" of those divisors of $f$ which are belong to the residue class $h\;mod(Q)$,\;i.e.,
\[\sigma_a(f;h,Q)=\sum_{\substack{d|f,d\text{\;monic}\\d\equiv h \text{\;mod}(Q)}}|d|^a.\]
Then
\[\sigma_a(f;Q)=\sum_{\substack{h,deg(h)\leq deg(Q)\\(h,Q)=1}}\sigma_a(f;h,Q).\]
The distribution of divisors of monic polynomial in residue classes can be described by the variance
\[\mathbb{V}[\sigma_a(f;\circ,Q)]:=\frac{1}{\Phi(Q)}\sum_{\substack{h,deg(h)\leq deg(Q)\\(h,Q)=1}}\left(\sigma_a(f;h,Q)-\frac{\sigma_a(f;Q)}{\Phi(Q)}\right)^2.\]
\begin{lem}\label{l-5.1}We have
\[\sum_{\substack{h,deg(h)\leq deg(Q)\\(h,Q)=1}}\left(\sigma_a(f;h,Q)-\frac{\sigma_a(f;Q)}{\Phi(Q)}\right)^2=\frac{1}{\Phi(Q)}\sum_{\chi\neq\chi_0}\sigma_a(f,\chi)\sigma_a(f,\overline{\chi}).\tag{17}\]
\begin{proof}We calculate the right side of equation (17).\;We can obtain that
\begin{equation*}
\begin{aligned}
  & \frac{1}{\Phi(Q)}\sum_{\chi\neq\chi_0}\sigma_a(f,\chi)\sigma_a(f,\overline{\chi}) \\
  & =\frac{1}{\Phi(Q)}\sum_{\chi\neq\chi_0}\sum_{\substack{d_1,d_2|f\\d_1,d_2 \text{\;monic}}}\chi(d_1)\overline{\chi}(d_2)|d_1|^a|d_2|^a\\
  & =\sum_{\substack{d_1,d_2|f\\d_1,d_2 \text{\;monic}}}\left(\frac{1}{\Phi(Q)}\sum_{\chi\neq\chi_0}\chi(d_1)\overline{\chi}(d_2)\right)|d_1|^a|d_2|^a\\
  & =\sum_{\substack{d_1,d_2|f\text{\;monic}\\d_1\equiv d_2\text{\;mod}(Q)}}|d_1|^a|d_2|^a-\sigma_a(f;Q)^2\\
  & =\sum_{\substack{h,deg(h)\leq deg(Q)\\(h,Q)=1}}\sigma_a(f;h,Q)^2-\sigma_a(f;Q)^2.\\
 \end{aligned}
\end{equation*}
Note that
\begin{equation*}
\begin{aligned}
& \sum_{\substack{h,deg(h)\leq deg(Q)\\(h,Q)=1}}\left(\sigma_a(f;h,Q)-\frac{\sigma_a(f;Q)}{\Phi(Q)}\right)^2\\
& =\sum_{\substack{h,deg(h)\leq deg(Q)\\(h,Q)=1}}\sigma_a(f;h,Q)^2+\frac{1}{\Phi(Q)}{\sigma_a(f;Q)}^2-\frac{2\sigma_a(f;Q)}{\Phi(Q)}\sigma_a(f;Q)\\
& =\sum_{\substack{h,deg(h)\leq deg(Q)\\(h,Q)=1}}\sigma_a(f;h,Q)^2-\sigma_a(f;Q)^2.\\
\end{aligned}
\end{equation*}
This completes the proof.
\end{proof}
\end{lem}

A simple calculation shows the following Lemma and we obmit the proof.
\begin{lem}\label{l-3.2}Let $u,p,q$ be any complex numbers such that the left side of equation (18) is convergent,\;then we have
\[\sum_{v\geq 0}u^v\left(\sum_{0\leq j\leq v}p^j\right)\left(\sum_{0\leq j\leq v}q^j\right)=\frac{1-pqu^2}{(1-upq)(1-up)(1-uq)(1-u)}.\tag{18}\]
\end{lem}

\begin{thm}\label{t-5.3}For integer $n$ and $\min\{\mathbf{R}(s),\mathbf{R}(s-a),\mathbf{R}(s-b),\mathbf{R}(s-a-b),\mathbf{R}(2s-a-b)\}>1$,\;we have
\begin{equation*}
\begin{aligned}
&\sum_{f\text{\;monic}}\frac{y^{\Omega(f)}}{|f|^s}\mathbb{V}[\sigma_a(f;\circ,Q)]\\
&=\frac{1}{\Phi(Q)^2}\frac{\zeta_A(s,y)L(s-2a,\chi_0,y)}{L(2s-2a,\chi_0,y^2)}\sum_{\chi\neq\chi_0}L(s-a,\chi,y)L(s-a,\overline{\chi},y).\\
\end{aligned}
\end{equation*}
\end{thm}
\begin{proof}From Lemma \ref{l-3.2},\;we can get
\begin{equation*}
\begin{aligned}
  &  \sum_{f\text{\;monic}}\frac{y^{\Omega(f)}}{|f|^s}\sigma_a(f,\chi_1)\sigma_b(f,\chi_2)\\
  & =\prod_P\left(\sum_{v\geq 0}\frac{y^{\Omega(P^v)}}{|P|^{vs}}\sum_{0\leq i,j\leq v}\chi_1(P^i)\chi_2(P^j)|P|^{ai}|P|^{bj}\right)\\
  & =\prod_P\left(1-\frac{\chi_1(P)\chi_2(P)y^{2}}{|P|^{2s-a-b}}\right)\left(1-\frac{y}{|P|^s}\right)^{-1}\\
  &  \;\;\;\;\left(1-\frac{\chi_1(P)y}{|P|^{s-a}}\right)^{-1}
     \left(1-\frac{\chi_2(P)y}{|P|^{s-b}}\right)^{-1}
     \left(1-\frac{\chi_1(P)\chi_2(P)y}{|P|^{s-a-b}}\right)^{-1} \\
  & =\zeta_A(s,y)\frac{L(s-a,\chi_1,y)L(s-b,\chi_2,y)L(s-a-b,\chi_1\chi_2,y)}{L(2s-a-b,\chi_1\chi_2,y^2)},\\
\end{aligned}
\end{equation*}
provide
  \[\min\{\mathbf{R}(s),\mathbf{R}(s-a),\mathbf{R}(s-b),\mathbf{R}(s-a-b),\mathbf{R}(2s-a-b)\}>1.\]
From Lemma \ref{l-5.1},\;we have
\begin{equation*}
\begin{aligned}
&\sum_{f\text{\;monic}}\frac{y^{\Omega(f)}}{|f|^s}\mathbb{V}[\sigma_a(f;\circ,Q)]\\
&=\frac{1}{\Phi(Q)}\sum_{f\text{\;monic}}\frac{y^{\Omega(f)}}{|f|^s}\sum_{\substack{h,deg(h)\leq deg(Q)\\(h,Q)=1}}\left(\sigma_a(f;h,Q)-\frac{\sigma_a(f;Q)}{\Phi(Q)}\right)^2,\\
&=\frac{1}{\Phi(Q)^2}\sum_{f\text{\;monic}}\frac{y^{\Omega(f)}}{|f|^s}\sum_{\chi\neq\chi_0}\sigma_a(f,\chi)\sigma_a(f,\overline{\chi}).
\end{aligned}
\end{equation*}
Then the result follows.
\end{proof}

Let $a=0$,\;then $\sigma_a(f;Q)$ be the numbers of  those divisors of $f$ which are prime to $Q$ and $\sigma_a(f;Q)$ be the numbers of those divisors of $f$ which belong to the residue class $h\;\text{mod\;}(Q)$.\;Let $\tau(f;h,Q)=\sigma_0(f;h,Q)$ and $\tau(f;Q)=\sigma_0(f;Q)$.\;We have following theorem.

\begin{thm}For integer $n$,\;we have
\begin{equation*}
\begin{aligned}
&\sum_{\substack{f\text{\;monic}\\deg(f)=n}}\mathbb{V}[\tau(f;\circ,Q)]=\frac{1}{\Phi(Q)}\sum_{\substack{f\text{\;monic}\\deg(f)=n}}\sum_{\substack{h,deg(h)\leq deg(Q)\\(h,Q)=1}}\left(\tau(f;h,Q)-\frac{\tau(f;Q)}{\Phi(Q)}\right)^2\\
&=A_{Q}(n+1)q^{n+1}+B_{Q}q^{n+1}+O_{Q,\varepsilon}(q^{n\varepsilon}(q^{1-\varepsilon}-1)^{-2}),
\end{aligned}
\end{equation*}
where
\[A_{Q}:=\frac{q-1}{q^2\Phi(Q)^2}\prod_{P|Q}\frac{1}{1+\frac{1}{|P|}}\sum_{\chi\neq\chi_0}|L(1,\chi)|^2,\]
\[B_{Q}=\frac{\prod\limits_{P|Q}(1+\frac{1}{|P|})^{-1}}{q^2\Phi(Q)^2}\left(\frac{q-1}{\log q}\left(2\sum_{\chi\neq\chi_0}|L(1,\chi)|^2\mathbf{R}\left(\frac{L'(1,\chi)}{L(1,\chi)}\right)+\sum_{P|Q}\frac{\log |P|}{|P|+1}\right)+2\right).\]
\end{thm}
\begin{proof}From Theorem \ref{t-5.3} with $y=1$ and equation (6),\;
the associated Drichlet series is
\begin{equation*}
\begin{aligned}
  & \frac{1}{\Phi(Q)^2}\frac{\zeta_A(s)L(s,\chi_0)}{L(2s,\chi_0)}\sum_{\chi\neq\chi_0}L(s,\chi)L(s,\overline{\chi})\\
  & =\frac{1}{\Phi(Q)^2}\left(\prod_{P|Q}\frac{1}{1+\frac{1}{|P|^s}}\right)\frac{\zeta^2_A(s)}{\zeta_A(2s)}\sum_{\chi\neq\chi_0}L(s,\chi)L(s,\overline{\chi}).\\
\end{aligned}
\end{equation*}
Thus,\;let $u=q^{-s}$,\;for $|u|<\frac{1}{q}$,\;we have
\[\sum_{n\geq 0}\left(\sum_{\substack{f\text{\;monic}\\deg(f)=n}}\mathbb{V}[\tau(f;\circ,Q)]\right)u^n=\frac{g_Q(u)}{\Phi(Q)^2}\frac{\zeta^2(u)}{\zeta(u^2)}\sum_{\chi\neq\chi_0}\mathcal{L}(u,\chi)\mathcal{L}(u,\overline{\chi}),\]
where $g_Q(u)=\prod\limits_{P|Q}\frac{1}{1+u^{deg(P)}}$.\;Let $0< R<\frac{1}{q}$ be arbitrary,\;by the Cauchy integral formula,\;we have
\[\sum_{\substack{f\text{\;monic}\\deg(f)=n}}\mathbb{V}[\tau(f;\circ,Q)]=\frac{1}{2\pi i\Phi(Q)^2}\int_{C_R}\frac{g_Q(z)}{z^{n+1}}\frac{1-qz^2}{(1-qz)^2}\sum_{\chi\neq\chi_0}\mathcal{L}(z,\chi)\mathcal{L}(z,\overline{\chi})dz.\]
Let $\rho<1$ be a constant,\;consider another region $\Omega_3$ surrounded by contour $\gamma_3=C_{\rho}\cup C_R$.\;By Cauchy's residue theorem,\;we can obtain that
\[\frac{1}{2\pi i}\int_{C_\rho}+\int_{C_R^{-}}F(z)dz=\operatorname*{Res}\limits_{z=\frac{1}{q}}F(z),\]
where $F(z)=\frac{g_Q(z)}{z^{n+1}\Phi(Q)^2}\frac{1-qz^2}{(1-qz)^2}\sum\limits_{\chi\neq\chi_0}\mathcal{L}(z,\chi)\mathcal{L}(z,\overline{\chi})$.
  Next we turn to calculate the integral along each of the curves.\;

We calculate the integral over $C_\rho$ at first.\;It follows from (7) that
\begin{equation*}
\begin{aligned}
  &\frac{1}{2\pi i\Phi(Q)^2}\int_{C_\rho}\frac{g_Q(z)}{z^{n+1}}\frac{1-qz^2}{(1-qz)^2}\sum\limits_{\chi\neq\chi_0}\mathcal{L}(z,\chi)\mathcal{L}(z,\overline{\chi})dz\\
  & \leq\frac{1}{2\pi\Phi(Q)\rho^{n+1}}\int_{C_\rho}|g_Q(z)|\frac{1+q\rho^2}{(q\rho-1)^2}(1+\sqrt{q}\rho)^{2m}dz\\
  & =O_{Q,\rho}\left(\frac{1+q\rho^2}{\rho^{n}(q\rho-1)^2}(1+\sqrt{q}\rho)^{2m}\right).
\end{aligned}
\end{equation*}
  Note that $z=\frac{1}{q}$ is a double pole of $F(z)$,\;so
\begin{equation*}
\begin{aligned}
&q^2\Phi(Q)^2\operatorname*{Res}\limits_{z=\frac{1}{q}}F(z)=\lim_{z\rightarrow\frac{1}{q}}\frac{d}{dz}\left(\frac{g_Q(z)(1-qz^2)}{z^{n+1}}\sum_{\chi\neq\chi_0}\mathcal{L}(z,\chi)\mathcal{L}(z,\overline{\chi})\right)\\
& =\left(\left(\frac{q-1}{q}g_{Q}'(\frac{1}{q})-2g_{Q}(\frac{1}{q})\right)q^{n+1}-g_{Q}(\frac{1}{q})\frac{(n+1)(q-1)}{q}q^{n+2}\right)\sum_{\chi\neq\chi_0}|L(1,\chi)|^2\\
& +\frac{q-1}{q}g_{Q}(\frac{1}{q})q^{n+1}\sum_{\chi\neq\chi_0}(\mathcal{L}(z,\chi)\mathcal{L}(z,\overline{\chi}))'\big|_{z=\frac{1}{q}}\\
&  =\left(\left(\frac{q-1}{q}g_{Q}'(\frac{1}{q})-2g_{Q}(\frac{1}{q})\right)q^{n+1}-g_{Q}(\frac{1}{q})(n+1)(q-1)q^{n+1}\right)\sum_{\chi\neq\chi_0}|L(1,\chi)|^2\\
& -2g_{Q}(\frac{1}{q})\frac{q-1}{\log q}q^{n+1}\sum_{\chi\neq\chi_0}|L(1,\chi)|^2\mathbf{R}\left(\frac{L'(1,\chi)}{L(1,\chi)}\right)\\
&=\prod_{P|Q}\frac{1}{1+\frac{1}{|P|}}\left(\left(-\frac{q-1}{\log q}\sum_{P|Q}\frac{\log |P|}{|P|+1}-2\right)q^{n+1}-(q-1)(n+1)q^{n+1}\right)\sum_{\chi\neq\chi_0}|L(1,\chi)|^2\\
& -2\prod_{P|Q}\frac{1}{1+\frac{1}{|P|}}\frac{q-1}{\log q}q^{n+1}\sum_{\chi\neq\chi_0}|L(1,\chi)|^2\mathbf{R}\left(\frac{L'(1,\chi)}{L(1,\chi)}\right).\\
\end{aligned}
\end{equation*}
Hence,\;we can get
\[\operatorname*{Res}\limits_{z=\frac{1}{q}}F(z)=-(A_{Q}(n+1)q^{n+1}+B_{Q}q^{n+1}).\]
Thus
\[\sum_{\substack{f\text{\;monic}\\deg(f)=n}}\mathbb{V}[\tau(f;\circ,Q)]=A_{Q}(n+1)q^{n+1}+B_{Q}q^{n+1}+O_{Q,\rho}\left(\frac{1+q\rho^2}{\rho^{n}(q\rho-1)^2}(1+\sqrt{q}\rho)^{2m}\right).\]

 Next,\;let $\rho=\frac{1}{q^\varepsilon}$($0<\varepsilon<1$),\;we have
\[O_{Q,\rho}\left(\frac{1+q\rho^2}{\rho^{n}(q\rho-1)^2}(1+\sqrt{q}\rho)^{2m}\right)=O_{Q,\varepsilon}(q^{n\varepsilon}(q^{1-\varepsilon}-1)^{-2}).\]
This completes the proof of Theorem 5.4.
\end{proof}

We now evaluate (5) by Selberg-Delange method.\;
\begin{thm}Let $y$ be a complex number such that $|y|\leq \rho$ and $\rho<1$.\;If $\mathbf{R}(y)<\frac{1}{2}$,\;then
\begin{equation*}
\begin{aligned}
&\sum_{\substack{f\text{\;monic}\\deg(f)=n}}y^{\Omega(f)}\mathbb{V}[\tau(f;\circ,Q)]
=H_1\left(\frac{1}{q},y\right)\frac{\kappa(2y)q^{n}}{\Gamma(2y)n^{1-2y}}+O_{\delta,Q,\rho}\left(\frac{q^{n}}{n^{2-2\mathbf{R}(y)}}\right)\\
\end{aligned}
\end{equation*}
for integer $n>3$,\;where
\[H_1(u,y)=\left(\prod\limits_{P|Q}\frac{1}{1+yu^{deg(P)}}\right)\frac{(1-qu^2)^{y}\mathcal{N}^2(u,y)}{\Phi(Q)^2\mathcal{N}(u^2,y)}\sum_{\chi\neq\chi_0}\mathcal{L}(u,\chi,y)\mathcal{L}(u,\overline{\chi},y).\]
 If $y=\frac{1}{2}$,\;then for integer $n$,
\begin{equation*}
\begin{aligned}
\sum_{\substack{f\text{\;monic}\\deg(f)=n}}\frac{1}{2^{\Omega(f)}}\mathbb{V}[\tau(f;\circ,Q)]
&=\frac{q^{n+1}}{\Phi(Q)^2}\prod\limits_{P|Q}\frac{1}{1+\frac{1}{2|P|}}\frac{N^2(1,\frac{1}{2})}{\zeta_A(2,\frac{1}{2})}\sum_{\chi\neq\chi_0}|L(1,\chi,\frac{1}{2})|^2\\
&+O_{\delta,Q}\left(q^{(\frac{1}{2}+\delta)n}\right).\\
\end{aligned}
\end{equation*}
\end{thm}

\begin{proof}From Theorem \ref{t-5.3},
the Drichlet series associates to it is
\begin{equation*}
\begin{aligned}
  & \frac{1}{\Phi(Q)^2}\frac{\zeta_A(s,y)L(s,\chi_0,y)}{L(2s,\chi_0,y^2)}\sum_{\chi\neq\chi_0}L(s,\chi,y)L(s,\overline{\chi},y)\\
  & =\frac{1}{\Phi(Q)^2}\left(\prod_{P|Q}\frac{1}{1+\frac{y}{|P|^s}}\right)\frac{\zeta^2_A(s,y)}{\zeta_A(2s,y)}\sum_{\chi\neq\chi_0}L(s,\chi,y)L(s,\overline{\chi},y).\\
\end{aligned}
\end{equation*}
Thus,\;let $u=q^{-s}$,\;for $|u|<\frac{1}{q}$,\;we have
\[\sum_{n\geq 0}\left(\sum_{\substack{f\text{\;monic}\\deg(f)=n}}y^{\Omega(f)}\mathbb{V}[\tau(f;\circ,Q)]\right)u^n=\frac{\widehat{g}_Q(u)}{\Phi(Q)^2}\frac{\zeta^2(u,y)}{\zeta(u^2,y)}\sum_{\chi\neq\chi_0}\mathcal{L}(u,\chi,y)\mathcal{L}(u,\overline{\chi},y),\]
where $\widehat{g}_Q(u)=\prod\limits_{P|Q}\frac{1}{1+yu^{deg(P)}}$.\;Let $0< R<\frac{1}{q}$ be arbitrary,\;by the Cauchy integral formula,\;we have
\[\sum_{\substack{f\text{\;monic}\\deg(f)=n}}y^{\Omega(f)}\mathbb{V}[\tau(f;\circ,Q)]=\frac{1}{2\pi i}\int_{C_R}\frac{H(z,y)}{z^{n+1}}dz,\]
where
\begin{equation*}
\begin{aligned}
H(z,y)&=\frac{\widehat{g}_Q(z)}{\Phi(Q)^2}\frac{\zeta^2(z,y)}{\zeta(z^2,y)}\sum_{\chi\neq\chi_0}\mathcal{L}(z,\chi,y)\mathcal{L}(z,\overline{\chi},y)\\
      &=\frac{\widehat{g}_Q(z)}{\Phi(Q)^2}\frac{\zeta(z)^{2y}\mathcal{N}^2(z,y)}{\zeta(z^2)^y\mathcal{N}(z^2,y)}\sum_{\chi\neq\chi_0}\mathcal{L}(z,\chi,y)\mathcal{L}(z,\overline{\chi},y)\\
      &=\frac{\widehat{g}_Q(z)}{\Phi(Q)^2}\frac{(1-qz^2)^{y}\mathcal{N}^2(z,y)}{(1-qz)^{2y}\mathcal{N}(z^2,y)}\sum_{\chi\neq\chi_0}\mathcal{L}(z,\chi,y)\mathcal{L}(z,\overline{\chi},y).
\end{aligned}
\eqno{(19)}
\end{equation*}

$\mathbf{Case \;1:}$\;$0<\mathbf{R}(y)<\frac{1}{2}$.
Considering the contours $\gamma_1$ of Figure 1,\;we can obtain that:
\[\frac{1}{2\pi i}\int_{\widetilde{C}_R^-}+\int_{\Gamma_1}+\int_{\widetilde{C}_{q^{-\frac{1}{2}-\delta}}}+\int_{C_\epsilon^-}+\int_{\Gamma_2}\frac{H(z,y)}{z^{n+1}}dz=0.\]
We estimate the integrals over $\widetilde{C}_{q^{-\frac{1}{2}-\delta}}\triangleq \widetilde{C}$ and $\widetilde{C}_\epsilon^-$ at first.\;From (9) and (10),\;we have
\begin{equation*}
\begin{aligned}
&\frac{1}{2\pi i}\int_{\widetilde{C}}\frac{H(z,y)}{z^{n+1}}dz\ll_{\delta,Q,\rho}\int_{\widetilde{C}}\frac{1}{z^{n+1}}\frac{|1-qz^2|^{\mathbf{R}(y)}}{|1-qz|^{2\mathbf{R}(y)}}(1+\sqrt{q}|z|)^{2\mathbf{R}(y)m}dz \\
&=q^{(\frac{1}{2}+\delta)n}\frac{(1+q^{-2\delta})^{\mathbf{R}(y)}}{(q^{\frac{1}{2}-\delta}-1)^{2\mathbf{R}(y)}}(1+q^{-\delta})^{2\mathbf{R}(y)m}\ll_{\delta,Q,\rho}q^{(\frac{1}{2}+\delta)n}.\\
\end{aligned}\eqno{(20)}
\end{equation*}
Let $H_1(z,y)=H(z,y)(1-qz)^{2y}$,\;then
\[\frac{1}{2\pi i}\int_{\widetilde{C}_\epsilon^-}\frac{H_1(z,y)}{(1-qz)^{2y}z^{n+1}}dz=\frac{\epsilon^{1-2y}}{2\pi q^{2y}}\int_{\alpha}^{2\pi-\alpha}\frac{e^{(1-2y)i\theta-2\pi yi}H_1(\frac{1}{q}+\epsilon e^{i\theta})}{(\frac{1}{q}+\epsilon e^{i\theta})^{n+1}}d\theta\]
\[\ll_{\delta,Q,\rho,n}\epsilon^{1-2\mathbf{R}(y)}.\hspace{10pt}\]

If $0<\mathbf{R}(y)<\frac{1}{2}$,\;then this integral tends to zero as $\epsilon\rightarrow0$.\;Letting the $\Gamma_1$,\;$\Gamma_2$ onto the real line and $\epsilon\rightarrow0$ together,\;we have
\begin{equation*}
\begin{aligned}
\frac{1}{2\pi i}\int_{C_R}\frac{H(z,y)}{z^{n+1}}dz
&=\frac{q^{-2y}}{2\pi i}\int_{\frac{1}{q}}^{\frac{1}{q^{\frac{1}{2}+\delta}}}\frac{H_1(u,y)}{u^{n+1}}e^{-2y(\log(u-\frac{1}{q})-i\pi)}du\\
&-\frac{q^{-2y}}{2\pi i}\int_{\frac{1}{q}}^{\frac{1}{q^{\frac{1}{2}+\delta}}}\frac{H_1(u,y)}{u^{n+1}}e^{-2y(\log(u-\frac{1}{q})+i\pi)}du\\
&+O_{\delta,Q,\rho}\left(q^{(\frac{1}{2}+\delta)n}\right).\\
\end{aligned}
\end{equation*}
Thus,
\begin{equation*}
\begin{aligned}
&\sum_{\substack{f\text{\;monic}\\deg(f)=n}}y^{\Omega(f)}\mathbb{V}[\tau(f;\circ,Q)]\\
&=\frac{\sin(2y\pi)}{q^{2y}\pi}\int_{\frac{1}{q}}^{\frac{1}{q^{\frac{1}{2}+\delta}}}\frac{H_1(u,y)}{u^{n+1}}(u-\frac{1}{q})^{-2y}du+O_{\delta,Q,\rho}\left(q^{(\frac{1}{2}+\delta)n}\right)\\
&=\frac{\sin(2y\pi)}{q^{2y}\pi}\int_{0}^{\omega_0}\frac{H_1(\frac{1}{q}+\omega,y)}{(\frac{1}{q}+\omega)^{n+1}}\omega^{-2y}d\omega+O_{\delta,Q,\rho}\left( q^{n(\frac{1}{2}+\delta)}\right).\\
\end{aligned}
\end{equation*}
Note that
\[\frac{H_1(\frac{1}{q}+\omega,y)}{\frac{1}{q}+\omega}-qH_1\left(\frac{1}{q},y\right)=\frac{\omega}{2\pi i}\int_{D}\frac{H_1(z,y)}{z(z-\frac{1}{q}-\omega)(z-\frac{1}{q})}dz=\omega k_1(\omega),\]
where $D$ is circle,\;centre $\frac{1}{q}+\frac{\omega_0}{2}$ and radius $\omega_0$.\;We can also obtain that
\begin{equation*}
\begin{aligned}
&\frac{\sin(y\pi)}{q^{y}\pi}\int_{0}^{\omega_0}\frac{H_1(\frac{1}{q}+\omega,y)}{(\frac{1}{q}+\omega)^{n+1}}\omega^{-2y}d\omega\\
&=qH_1\left(\frac{1}{q},y\right)\frac{\sin(2y\pi)}{q^{2y}\pi}\int_{0}^{\omega_0}\frac{\omega^{-2y}}{(\frac{1}{q}+\omega)^{n}}d\omega+O_{\delta,Q}\left(\int_{0}^{\omega_0}\frac{\omega^{1-2y}}{(\frac{1}{q}+\omega)^{n}}d\omega\right)\\
&=H_1\left(\frac{1}{q},y\right)\frac{\sin(2y\pi)}{\pi}q^{n}\frac{\Gamma(1-2y)\Gamma(n-1+2y)}{\Gamma(n)}+O_{\delta,Q,\rho}\left(\frac{q^{n}}{n^{2-2\mathbf{R}(y)}}\right)\\
&=H_1\left(\frac{1}{q},y\right)\frac{\kappa(2y)q^{n}}{\Gamma(2y)n^{1-2y}}+O_{\delta,Q,\rho}\left(\frac{q^{n}}{n^{2-2\mathbf{R}(y)}}\right).
\end{aligned}
\end{equation*}
as the proof of Theorem \ref{t-3.3}.

$\mathbf{Case \;2:}$\;$y=\frac{1}{2}$.\;We recall the function $H(z,y)$ of equation (19),\;it has a simple pole at $z=\frac{1}{q}$ with
residue
\[\frac{\widehat{g}_Q(\frac{1}{q})q^{n+1}}{\Phi(Q)^2}\frac{(1-\frac{1}{q})^{\frac{1}{2}}\mathcal{N}^2(\frac{1}{q},\frac{1}{2})}{\mathcal{N}(\frac{1}{q^2},\frac{1}{2})}\sum_{\chi\neq\chi_0}\mathcal{L}\left(\frac{1}{q},\chi,\frac{1}{2}\right)\mathcal{L}\left(\frac{1}{q},\overline{\chi},\frac{1}{2}\right)\]
\[=\frac{q^{n+1}}{\Phi(Q)^2}\prod\limits_{P|Q}\frac{1}{1+\frac{1}{2|P|}}\frac{N^2(1,\frac{1}{2})}{\zeta_A(2,\frac{1}{2})}\sum_{\chi\neq\chi_0}|L(1,\chi,\frac{1}{2})|^2.\hspace{25pt}\]
Thus,\;considering the contours $\gamma_2$ of Figure 2,\;we have
\begin{equation*}
\begin{aligned}
\sum_{\substack{f\text{\;monic}\\deg(f)=n}}\frac{1}{2^{\Omega(f)}}\mathbb{V}[\tau(f;\circ,Q)]
&=\frac{q^{n+1}}{\Phi(Q)^2}\prod\limits_{P|Q}\frac{1}{1+\frac{y}{|P|}}\frac{N^2(1,\frac{1}{2})}{\zeta_A(2,\frac{1}{2})}\sum_{\chi\neq\chi_0}|L(1,\chi,\frac{1}{2})|^2\\
&+O_{\delta,Q,\rho}\left(q^{(\frac{1}{2}+\delta)n}\right).\\
\end{aligned}
\end{equation*}
Then the result follows.
\end{proof}
Let
\[H_1\left(\frac{1}{q},y\right)\frac{\kappa(2y)}{\Gamma(2y)}=\sum_{r=1}^{+\infty}\widehat{A}_ry^r,\]
then we obtain the following theorem.

\begin{thm}
\[\sum_{\substack{f\text{\;monic}\\deg(f)=n,\Omega(f)=t}}\mathbb{V}[\tau(f;\circ,Q)]=\frac{q^n(2\log(n))^{t-1}}{n}\sum_{r=1}^{t}\frac{\widehat{A}_r(2\log(n))^{1-r}}{(t-r)!}+O_{\delta,Q,\rho}\left(\frac{\rho^{-n}q^n}{n^{2-\mathbf{R}{(y)}}}\right)\]
for any $|\rho|<1$.
\end{thm}

\section*{Acknowledgements}
The author are grateful to those of you who support to us.\;
\section*{References}

\bibliographystyle{plain}

\bibliography{1}

  \end{document}